\documentclass[11pt]{amsart}

\usepackage{amsmath,amssymb,amsthm, verbatim,hyperref}
\usepackage{fullpage}
\usepackage{xcolor}
\usepackage{url}
\usepackage{mathrsfs}
\usepackage[all]{xy}
  \SelectTips{cm}{10}
  \everyxy={<2.5em,0em>:}
\usepackage{tikz}
\usepackage{picinpar} 
\usepackage{multirow}
\usepackage[title]{appendix}

\DeclareMathOperator{\mult}{mult}

\DeclareMathOperator{\expected}{expected}

\usepackage{enumerate, amsmath, amsthm, amsfonts, amssymb,  mathrsfs}

\theoremstyle{plain}
  \newtheorem{lemma}[equation]{Lemma}
  \newtheorem{proposition}[equation]{Proposition}
  \newtheorem{theorem}[equation]{Theorem}
  \newtheorem{corollary}[equation]{Corollary}

\theoremstyle{definition}
  \newtheorem{definition}[equation]{Definition}

\theoremstyle{remark}
  \newtheorem{remark}[equation]{Remark}

\renewcommand{\thesection}{\arabic{section}}
\renewcommand{\theequation}{\thesection.\arabic{equation}}

 \DeclareFontFamily{U}{manual}{}
 \DeclareFontShape{U}{manual}{m}{n}{ <->  manfnt }{}
 \newcommand{\manfntsymbol}[1]{%
    {\fontencoding{U}\fontfamily{manual}\selectfont\symbol{#1}}}

\makeatletter
   \@addtoreset{section}{part}
   \@addtoreset{equation}{section}
   \@addtoreset{footnote}{section}

    {\hspace*{\fill}$\lrcorner$\endgraf\endgroup\end{trivlist}}
 
\makeatother

  \DeclareFontFamily{OT1}{pzc}{}
  \DeclareFontShape{OT1}{pzc}{m}{it}{<-> s * [1.100] pzcmi7t}{}
  \DeclareMathAlphabet{\mathpzc}{OT1}{pzc}{m}{it}

\newif\ifhascomments \hascommentstrue
\ifhascomments
  \newcommand{\david}[1]{{\color{red}[[\ensuremath{\bigstar\bigstar\bigstar} #1]]}}
  \newcommand{\matt}[1]{{\color{red}[[\ensuremath{\spadesuit\spadesuit\spadesuit} #1]]}}
  \newcommand{\austin}[1]{{\color{red}[[\ensuremath{\clubsuit\clubsuit\clubsuit} #1]]}}
\else
  \newcommand{\david}[1]{}
  \newcommand{\matt}[1]{}
  \newcommand{\austin}[1]{}
\fi

\DeclareMathOperator{\Endo}{\ensuremath{\mathcal{E}\kern-.125em\mathpzc{nd}}}

\DeclareMathOperator{\Hom}{\ensuremath{\mathcal{H}\kern-.125em\mathpzc{om}}}

\newcommand{\PP}{\mathbb{P}}

\newcommand{\QQ}{\mathbb Q}
\newcommand{\RR}{\mathbb R}

\newcommand{\ZZ}{\mathbb{Z}}

 \def\ari[#1]{\ar@{^(->}[#1]}
 \def\are[#1]{\ar[#1]^{\txt{\'et}}}
 \def\areh[#1]{\ar[#1]|{\txt{$H$-eq}}^{\txt{\'et}}}
 \def\ars[#1]{\ar@{->>}[#1]}
 \newcommand{\dplus}{\ar@{}[d]|{\mbox{$\oplus$}}}
 \newcommand{\dtimes}{\ar@{}[d]|{\mbox{$\times$}}}

\usepackage{amsthm}

\theoremstyle{plain}

\newtheoremstyle{named}{}{}{\itshape}{}{\bfseries}{.}{.5em}{\thmnote{#3 }#1}
\theoremstyle{named}

\usepackage{amsmath}

\DeclareMathOperator{\Conv}{Conv}

\theoremstyle{remark}

\def\Q{{\mathbb Q}}

\def\x{{\mathbf x}}

\title{On curves with high multiplicity on $\mathbb{P}(a,b,c)$ for $\min(a,b,c)\leq4$}

\author{David McKinnon}
\address{University of Waterloo \\
Department of Pure Mathematics \\
Waterloo, Ontario \\
Canada  N2L 3G1}

\email{dmckinnon@uwaterloo.ca}

\author{Rindra Razafy}
\address{University of Waterloo \\
Department of Pure Mathematics \\
Waterloo, Ontario \\
Canada  N2L 3G1}

\email{rrazafy@uwaterloo.ca}

\author{Matthew Satriano}
\address{University of Waterloo \\
Department of Pure Mathematics \\
Waterloo, Ontario \\
Canada  N2L 3G1}

\email{msatrian@uwaterloo.ca}

\author{Yuxuan Sun}
\address{University of Waterloo \\
Department of Pure Mathematics \\
Waterloo, Ontario \\
Canada  N2L 3G1}

\email{y376sun@uwaterloo.ca}

\thanks{The first author and third authors were partially supported by a Discovery Grant from the National Science and Engineering Board of Canada.  The second and fourth authors were supported by an Undergraduate Student Research Award from the National Science and Engineering Board of Canada.}

\begin{document}

\begin{abstract}
On a weighted projective surface $\PP(a,b,c)$ with $\min(a,b,c)\leq 4$, we compute lower bounds for the {\em effective threshold} of an ample divisor, in other words, the highest multiplicity a section of the divisor can have at a specified point.  We expect that these bounds are close to being sharp.  This translates into finding divisor classes on the blowup of $\PP(a,b,c)$ that generate a cone contained in, and probably close to, the effective cone.  
\end{abstract}

\maketitle

\section{Introduction}
Given a projective variety $X$ and a point $Q\in X$, it is, in general, a notoriously difficult problem to calculate the pseudo-effective cone of the blow-up $\textrm{Bl}_Q(X)$ in terms of the pseudo-effective cone of $X$. Even addressing the \emph{a priori} easier question of when $\textrm{Bl}_Q(X)$ is a Mori Dream Space, where $X=\PP(a,b,c)$ is a weighted projective surface and $Q$ is the identity of its torus, is already challenging and has a rich history \cite{Hun82,Cut91,Sri91,GNW94,CT15,GK16,He17,GGK20}. To gain information about the pseudo-effective cone of $\textrm{Bl}_Q(X)$, we consider the following quantity, cf.~\cite{Fuj92}.

\begin{definition}
Let $X$ be a projective variety defined over a field $k$, $D$ a $k$-rational $\Q$-divisor, and $Q$ a $k$-rational point of $X$.  Let $\pi$ be the blowup of $X$ at $Q$ and $E$ the exceptional divisor of $\pi$. We say the \emph{effective threshold} is
\[
\gamma_Q(D) := \sup\{\gamma>0\mid \pi^*(D)-\gamma E\,\,\mbox{is pseudo-effective}\}.
\]
\end{definition}

The quantity $\gamma_Q(D)$ can be reinterpreted concretely as follows. If there is a curve in the class of $D$ with multiplicity $m$ at $Q$, then $\gamma_Q(D)\geq m$. Conversely, if $\gamma_Q(D)=m$, then for all $\epsilon>0$, the class $\pi^*D-(m-\epsilon)E$ is pseudo-effective, so $D$ contains curves of multiplicity arbitrarily close to $m$, at least in a $\Q$-divisor sense. So computing $\gamma_Q(D)$ essentially amounts to computing
\[
\sup_{C,m}\left\{\frac{1}{m}\mult_Q(C)\right\}
\]
as $m$ varies through positive integers and $C$ varies through curves in the divisor class $mD$.

In this paper, we give characteristic-free lower bounds for $\gamma_Q(D)$ in the case where $X$ is the weighted projective surface $\PP(a,b,c)$ and $\min(a,b,c)\leq4$. In fact, we do more than this: we introduce a combinatorial quantity $\gamma_{\expected}$ which is a lower bound on $\gamma$, and compute $\gamma_{\expected}$ exactly. It is worth remarking here that although the motivation for studying $\gamma_Q$ is geometric, our lower bounds on $\gamma_Q$ also have consequences for Diophantine approximation problems related to generalizations of Roth's famous 1955 theorem \cite{roth-rat-approx}, see e.g., \cite[Theorem 3.3]{MR16} and \cite[Section 8]{MS20}.

In \cite{GGK20}, the authors make a series of detailed calculations closely related to what we compute in this paper.  In particular, they search the spaces of global sections of toric surfaces of Picard rank one for curves whose strict transforms have negative self-intersection upon blowing up a point.  If there is such a curve, then the pseudo-effective cone of the blowup will be finitely generated by the exceptional divisor of the blowup and another curve of negative self-intersection.

In this paper, we compute not only curves, but also the corresponding value of the effective threshold.  We do not prove that the curves we find are always generators of the pseudo-effective cone, but in most cases the value of $\gamma$ we compute is expected to be equal or very close to the actual value.  As the authors of \cite{GGK20} also point out, our quantity $\gamma_{expected}$ is expected to be very close to the actual value of $\gamma$.

Since $X=\PP(a,b,c)$ is a toric surface, if the point $Q$ does not lie in the main torus orbit $T$, then computing $\gamma_Q$ is generally straightforward, so we may assume that $Q$ lies in $T$. Furthermore, we can choose $a$, $b$, and $c$ to be pairwise coprime, with $a\leq b\leq c$. These inequalities are always strict unless $a=b=1$, in which case 
$\gamma$ can be computed directly.  Thus, we may assume that $a<b<c$.  Finally, since $X$ has Picard rank $1$, it suffices to compute $\gamma_Q(H)$, where $H$ is the generator of the Cartier class group.

Our first result concerns the case $a<4$ and serves as a warm-up to our main result.

\begin{proposition}
\label{prop:a-leq3-case}
Let $a<b<c$ be pairwise coprime, so we may write $c=pa+qb$ with $p,q\in\ZZ$ and $0\leq q<a$. Let $Q$ be in the torus of $\PP(a,b,c)$ and $H$ be the generator of the Cartier class group. Then
\[
\gamma_Q(H)\geq
\begin{cases}
(q+1)b, & p\geq 0\\
(a-1)b, & p<0\textrm{\ and\ }a\leq3.
\end{cases}
\]
\end{proposition}

Proposition \ref{prop:a-leq3-case} yields lower bounds on $\gamma_Q$ when $a\leq3$. Moving from $a\leq3$ to $a=4$ is significantly more involved. In order to state our results, we first discuss our main technique of proof. Note that if $m\in\QQ^+$ and $mH$ is a Weil divisor such that $h^0(mH)>{\nu+1\choose 2}$, then there is a global section $g$ of $mH$ that vanishes at $Q$ to order $\nu$. Writing $m=\frac{m_1}{m_2}$ with $m_1,m_2\in\ZZ^+$, we see $g^{m_2}\in H^0(X,m_1H)$ vanishes to order $\nu m_2$. By definition, it follows that $\gamma_Q(H)\geq\frac{\nu m_2}{m_1}=\frac{\nu}{m}$. This motivates the following definition.

\begin{definition}
\label{def:nu-and-gamma-exp}
For any Weil divisor $D$, let
\[
\nu(D) := \max\left\{ d \in \mathbb{Z}^+ \mid h^0(D) > \binom{d + 1}{2} \right\}.
\]
If $H$ denotes the generator of the Cartier class group of $\PP(a,b,c)$ with $a<b<c$, then let
\[
\gamma_{\expected}(H) := \sup\left\{\frac{\nu(mH)}{m}\mid m\in\frac{1}{b}\ZZ^+\cup\frac{1}{c}\ZZ^+\right\}.
\]
\end{definition}

We can now state our main result. Recall that Proposition \ref{prop:a-leq3-case} already yields lower bounds on $\gamma_Q$ when $p\geq0$ in general, so we turn to the case $p<0$.

\begin{theorem}
\label{c:P4bc-conjs}
With notation and hypotheses as in Proposition \ref{prop:a-leq3-case}, assume $a=4$ and $p<0$. Then
\[
\gamma_Q(H)\geq\gamma_{\expected}(H) = \frac{\nu(D_0)}{m_0},
\] 
where $D_0$ and $\nu(D_0)$ are computed exactly as follows. Given our constraints, we have $2<\frac{b}{-p}<\frac{16}{3}$.  Divide the interval $[2,\frac{16}{3}]$ into a countably infinite sequence of intervals of the form
\[
I_k:=\left[\frac{16(k+1)^2}{8(k+1)^2-4(k+1)-1},\frac{16k^2}{8k^2-4k-1}\right]
\]
with $k\in\ZZ^+$. Then the class of $D_0\sim m_0 H$ is given as follows, depending on the value of $\frac{b}{-p} \in I_k$: 
\begin{enumerate}
\item If $\frac{b}{-p} \in I'_{k,-} := \left[\frac{16(k+1)^2}{8(k+1)^2-4(k+1)-1},\frac{2k+1}{k}\right]$, then $D_0 \sim \frac{2k+3}{c}H$ with $\nu(D_0)=4(k+1)$.
\item If $\frac{b}{-p} \in I'_{k,+} := \left[\frac{2k+1}{k},\frac{4(2k+1)^2}{8k^2+4k-1}\right]$, then $D_0 \sim \frac{2k+1}{b}H$ with $\nu(D_0)=4(k+1)$.
\item If $\frac{b}{-p} \in I''_{k,-} := \left[\frac{4(2k+1)^2}{8k^2+4k-1}, \frac{4k}{2k-1}\right]$, then $D_0 \sim \frac{k+1}{c}H$ with $\nu(D_0)=2k+1$.
\item If $\frac{b}{-p} \in I''_{k,+} := \left[\frac{4k}{2k-1},\frac{16k^2}{8k^2-4k-1}\right]$, then $D_0 \sim \frac{k}{b}H$ with $\nu(D_0)=2k+1$.
\end{enumerate}
%
\end{theorem}

\begin{remark}
It is straightforward to check from Theorem \ref{c:P4bc-conjs} that $\gamma_{\expected}(D_0)$ is a continuous function of $\frac{b}{-p}$.
\end{remark}

\begin{remark}
The quantity $\gamma_{\expected}(H)$ is the lower bound for $\gamma(H)$ obtained by simple linear algebra: the vanishing to order $n$ of a section of $H$ is equivalent to the vanishing of ${n+1\choose 2}$ linear forms on the space of sections of $H$.  We therefore have $\gamma(H)\geq\gamma_{\expected}(H)$ trivially.

However, one also expects that the two quantities are not so different.  In particular, if $\gamma(H)>\gamma_{\expected}(H)$ at some point $Q$ in the main torus orbit, then there is a section $s$ of some multiple $mH$ of $H$ that has an order of vanishing that is greater than $\nu(mH)$ at $Q$.  For any element $\sigma$ of the torus, the section $\sigma(s)$ has unusually high order of vanishing at $\sigma(Q)$, so for {\em every} point of the main orbit, there is a section of $mH$ that has unusually high order of vanishing there.  This is unlikely -- though not downright impossible -- and so one expects the two quantities to be close.

\end{remark}

The rest of the paper is organized as follows.  Section \ref{sec:prelim-reds} proves Proposition \ref{prop:a-leq3-case} and describes some preliminary reductions for Theorem~\ref{c:P4bc-conjs}. Section \ref{sec:Ehrhart} computes the main terms in the count of global sections of multiples of $H$.  Section \ref{sec:last2terms} then begins the process of bounding the error terms, and Section \ref{sec:proof-of-thm} finishes the proof of Theorem~\ref{c:P4bc-conjs}.

\section*{Acknowledgments}
We are grateful to Kalle Karu for many enlightening email exchanges. This paper is the outcome of an NSERC-USRA project; we thank NSERC for their support.

\section{Proof of Proposition \ref{prop:a-leq3-case} and preliminary reductions}
\label{sec:prelim-reds}

Throughout this paper, we let $x$, $y$, and $z$ be the weighted projective coordinates on $\mathbb{P}(a,b,c)$ weights $a$, $b$, and $c$, respectively. We let $D_x$, $D_y$, and $D_z$ denote the Weil divisors defined by the vanishing of $x$, $y$, and $z$, respectively. We let $H$ denote the generator of the Cartier class group, so we have $H\sim bcD_x\sim acD_y\sim abD_z$. Given any Weil divisor $D$, we let $P_D\subset\RR^2$ be the associated polytope with the property that $h^0(D)=|P_D\cap\ZZ^2|$. We sometimes abusively note $|P_D\cap\ZZ^2|$ by $|P_D|$.

After a preliminary lemma, we prove Proposition \ref{prop:a-leq3-case-more-general} which is a slightly more general version of Proposition \ref{prop:a-leq3-case}.

\begin{lemma}
\label{l:qneq1}
With notation and hypotheses as in Proposition \ref{prop:a-leq3-case}, if $p<0$, then $q\neq1$. In particular, if $p<0$ and $a\leq3$, then $a=3$ and $q=2$.
\end{lemma}
\begin{proof}
If $q=1$, then $c=pa+b<b$ which is a contradiction. If $p<0$ and $a\leq3$, then $q=0$ or $q\geq2$. The former case cannot occur as it implies $c=pa$ and hence $p>0$. The latter case implies $2\leq q\leq a-1$ so $q=2$ and $a=3$.
\end{proof}

\begin{proposition}
\label{prop:a-leq3-case-more-general}
With notation and hypotheses as in Proposition \ref{prop:a-leq3-case}, we have
\[
\gamma_Q(H)\geq
\begin{cases}
(q+1)b, & p\geq 0\\
(a-1)b, & p<0,\ q=a-1, \textrm{\ and\ } \frac{-pa}{b}\leq1.
\end{cases}
\]
Furthermore, if $a\leq3$ and $p<0$, then $q=a-1$ and $\frac{-pa}{b}\leq1$ automatically hold.
\end{proposition}
\begin{proof}
Note that since $a$, $b$, and $c$ are pairwise coprime, $p\neq0$. First suppose $p>0$. Then the polytope $P_{aD_z}$ is the convex hull of $0$, $(q,-a)$, and $(\frac{-pa}{b},-a)$, so it contains the triangle $T$ with vertices $0$, $(q,-a)$, $(0,-a)$. By Pick's Theorem, $1 + {q+2 \choose 2}\leq|T\cap\ZZ^2|\leq|P_{aD_z}\cap\ZZ^2|$, which implies $\nu(aD_z)\geq q+1$. Since $aD_z\sim\frac{1}{b}H$, we find $\gamma_Q(H)\geq (q+1)b$.

Next suppose $p<0$, $q=a-1$, and $\frac{-pa}{b}\leq1$. Notice that the polytope $P_{aD_z}$ is given by the vertices as above, and it contains the triangle $T$ with vertices $0$, $(q,-a)$, and $(1,-a)$ by $p < 0$ and $\frac{-pa}{b}\leq1$. From $q=a-1$ and Pick's Theorem, we have $1 + {a \choose 2}=|T\cap\ZZ^2|$, which, as in the previous paragraph, implies $\gamma_Q(H)\geq (a-1)b$.

Finally, we note that if $a\leq3$ and $p<0$, then Lemma \ref{l:qneq1} tells us $a=3$ and $q=2=a-1$. Then $b<c=pa+2b$ implies $\frac{-pa}{b}<1$.
\end{proof}

The rest of the paper is concerned with the proof of Theorem \ref{c:P4bc-conjs}. By Lemma \ref{l:qneq1}, since $p<0$ and $a,b,c$ are pairwise coprime, we must have
\[
q=3=a-1.
\]
We begin by analyzing $\nu(D_0)$.

\begin{proposition}
\label{prop:nu-of-winning-poly}
With notation and hypotheses as in Theorem \ref{c:P4bc-conjs}, if $\frac{b}{-p}\in I'_k := I'_{k,+} \cup I'_{k,-}$, resp.~$I''_k := I''_{k,+} \cup I''_{k,-}$, and $D_0$ is as in the conclusion of the theorem, then $\nu(D_0)\geq 4(k+1)$, resp.~$2k+1$.
\end{proposition}
\begin{proof}
Let $m_0\in\ZZ$ be such that $D_0\sim\frac{m_0}{b}H\sim m_0cD_x$ or $D_0\sim\frac{m_0}{c}H\sim m_0bD_x$. In the former (respectively latter) case, $h^0(D_0)$ is given by the number of integer lattice points lying in the polytope
\[
P_{m_0cD_x} =\Conv\left((0,0),\,\left(-m_0\frac{c}{b},0\right),\,(-3m_0,4m_0)\right)
\]
respectively
\[
P_{m_0bD_x} = \Conv\left((0,0),\,(-m_0,0),\,\left(-3m_0\frac{b}{c},4m_0\frac{b}{c}\right)\right).
\]

First consider the case $\frac{b}{-p} \in I'_k$. As in Theorem \ref{c:P4bc-conjs}, $I'_{k,+}:=[\frac{2k+1}{k},\frac{4(2k+1)^2}{8k^2+4k-1}]$ and $I'_{k,-}:=[\frac{16(k+1)^2}{8(k+1)^2-4(k+1)-1},\frac{2k+1}{k}]$. 
If $\frac{b}{-p}\in I'_{k,+}$, then $D_0\sim m_0cD_x$ with $m_0=2k+1$ and if $\frac{b}{-p}\in I'_{k,-}$, then $D_0\sim m_0bD_x$ with $m_0=2k+3$. Let
\[
P' = \Conv((0,0),\,(-(2k+3),0),\,(-3(2k+1),4(2k+1))).
\]
We then have $P'\subset P$. Indeed, if $\frac{b}{-p}\in I'_{k,+}$, then the inclusion follows from $-(2k+1)\frac{c}{b} = -(2k+1)(4\frac{p}{b}+3) \leq -(2k+1)(4\frac{-k}{2k+1}+3) = -(2k+3)$. If $\frac{b}{-p}\in I'_{k,-}$, then the inclusion follows from $-3(2k+3)\frac{b}{c} \leq -3(2k+1)$ and $4(2k+3)\frac{b}{c} \geq 4(2k+1)$. So, in either case, we have
\[
|P'\cap\ZZ^2|\leq h^0(D_0).
\]
Note that the area of $P'$ is given by 
\[
A(P') = \frac{1}{2}(4(2k+1)(2k+3)) = 2(2k+1)(2k+3)
\]
and the number of lattice points on its boundary is given by 
\[
B(P') = (2k+3) + (2k+1) + 4 = 4k + 8.
\] 
Since $P'$ is a lattice polygon, applying Pick's Theorem, we have
\[
|P'\cap\ZZ^2| = A(P') + \frac{1}{2}B(P') + 1 = 8k^2 + 18k + 11 = \binom{4(k+1)+1}{2}+1,
\]
which shows $\nu(D_0) \geq \nu(P') = 4(k+1)$.

For $\frac{b}{-p}\in I''_k$, the same proof works when use
\[
P'' = \Conv((0,0),\,(-(k+1),0),\ (-3k,4k))
\]
in place of $P'$.
\end{proof}

Proposition \ref{prop:nu-of-winning-poly} therefore gives the lower bound
\[
\gamma_{\expected}(H) \geq \frac{\nu(D_0)}{n},
\]
where $D_0=nH$ is the divisor class described in Theorem~\ref{c:P4bc-conjs}. To obtain upper bounds, we introduce the following quantities and make use of the subsequent lemma. Let
\begin{align*}
\gamma_{\expected,b}(H) &:= \sup\left\{\frac{c}{m}\nu(D)\mid D=mbD_x\sim\frac{m}{c}H\right\}\\
\gamma_{\expected,c}(H) &:= \sup\left\{\frac{b}{m}\nu(D)\mid D=mcD_x\sim\frac{m}{b}H\right\}.
\end{align*}

\begin{lemma}
\label{l:red-to-bds-seqs}
Suppose $\mathbb{P}(4,b_0,c_0)$, $\mathbb{P}(4,b_L,c_L)$, and $\mathbb{P}(4,b_U,c_U)$ satisfy the hypotheses of Theorem \ref{c:P4bc-conjs}, except we need not assume $p<0$. Suppose $\frac{b_L}{-p_L} < \frac{b_0}{-p_0} < \frac{b_U}{-p_U}$ and let $H_0$, $H_L$, and $H_U$ denote the generators of the respective Cartier class groups. Then
\[
\gamma_{\expected,b}(H_0) \leq \gamma_{\expected,b}(H_L)\quad\quad\textrm{and}\quad\quad
\gamma_{\expected,c}(H_0) \leq \gamma_{\expected,c}(H_U).
\]
\end{lemma}
\begin{proof}
Since $c=pa+qb=4p+3b$, we see $\frac{b}{c} = \frac{1}{4\frac{p}{b}+3}$. As a result, $\frac{b_U}{c_U} < \frac{b_0}{c_0} < \frac{b_L}{c_L}$. It follows that
\[
P_0 := \Conv\left((0,0),(-1,0),\left(-3\frac{b_0}{c_0},4\frac{b_0}{c_0}\right)\right)\,\subset\,\Conv\left((0,0),(-1,0),\left(-3\frac{b_L}{c_L},4\frac{b_L}{c_L}\right)\right)=:P_L
\]
Since $P_0$, resp.~$P_L$, is the polytope of $bD_x$ on $\PP(4,b_0,c_0)$, resp.~$\PP(4,b_L,c_L)$, we have $\nu(nP_0) \leq \nu(nP_L)$ for all $n \geq 1$, and so $\gamma_{\expected,b}(H_0) \leq \gamma_{\expected,b}(H_L)$.

We obtain the inequality $\gamma_{\expected,c}(H_0) \leq \gamma_{\expected,c}(H_U)$ in a similar manner from the inclusion
\[
\Conv\left((0,0),\left(-\frac{c_0}{b_0},0\right),(-3,4)\right)\,\subset\,\Conv\left((0,0),\left(-\frac{c_U}{b_U},0\right),(-3,4)\right),
\]
the lefthand, resp.~righthand, side being the polytope of $cD_x$ on $\PP(4,b_0,c_0)$, resp.~$\PP(4,b_U,c_U)$.
\end{proof}

\begin{remark}
\label{rmk:red-to-bds-seqs}
To prove Theorem \ref{c:P4bc-conjs}, we apply Lemma \ref{l:red-to-bds-seqs} as follows. Let $I=[\frac{\beta_1}{\alpha_1},\frac{\beta_2}{\alpha_2}] = I'_{k, \pm}$ or $I''_{k, \pm}$ as in Proposition \ref{prop:nu-of-winning-poly}. First, consider the case $D_0 \sim m_0bD_x$. Fix any weighted projective space $\PP(4,b,c)$ for which $\frac{b}{-p}\in I$, and let $H$ be the generator of its Cartier class group. We must show $\gamma_{\expected}(H) = \frac{c\nu_0}{m_0}$, with $\nu_0 := \nu(D_0)$ given as in Theorem \ref{c:P4bc-conjs}. This may be done as follows: fix increasing sequences of positive integers $\{b_i\}_i$ and $\{-p_i\}_i$ for which 
\[
\alpha_1b_i-\beta_1(-p_i)=1
\quad\quad\textrm{and}\quad\quad
c_i := 4p_i+3b_i > b_i \textrm{ is such that 4, $b_i$, $c_i$ are pairwise coprime.}
\]
We may always find such sequences, since $\alpha_1,\beta_1 > 0$ are coprime in all cases listed in Theorem \ref{c:P4bc-conjs}. Let $H_i$ denote the generator of the Cartier class group of $\PP(4,b_i,c_i)$. Then for $i$ sufficiently large, $\frac{b_i}{-p_i}\in I$ is monotonically decreasing with $\frac{b_i}{-p_i}\to\frac{\beta_1}{\alpha_1}$. Similarly, fix increasing sequences of positive integers $\{b'_i\}_i$ and $\{-p'_i\}_i$ for which 
\[
\alpha_2b'_i-\beta_2(-p'_i)=-1
\quad\quad\textrm{and}\quad\quad
c'_i := 4p'_i+3b'_i > b'_i \textrm{ is such that 4, $b'_i$, $c'_i$ are pairwise coprime.}
\]
As above, such sequences always exist. Let $H'_i$ denote the generator of the Cartier class group of $\PP(4,b'_i,c'_i)$. Then for $i$ sufficiently large, $\frac{b'_i}{-p'_i}\in I$ is monotonically increasing with $\frac{b'_i}{-p'_i}\to\frac{\beta_2}{\alpha_2}$. By Lemma \ref{l:red-to-bds-seqs}, it then suffices to check the theorem on these sequences of weighted projective spaces, i.e.~first, we need to show
\[
\frac{\nu_0}{m_0}\geq\frac{1}{m}\nu\left(\frac{m}{c_i}H_i\right)
\]
for all $i$ sufficiently large and $m\in\ZZ^+$, with strict inequality whenever $m$ is not a multiple of $m_0$. Then, we need to show
\[
\frac{c'_i\nu_0}{m_0} > \frac{b'_i}{m}\nu\left(\frac{m}{b_i}H'_i\right)
\]
for all $i$ sufficiently large and $m\in\ZZ^+$. A similar strategy applies for the case $D_0 \sim m_0cD_x$. Since there are 4 types of intervals listed in Theorem \ref{c:P4bc-conjs}, and for each interval we must prove a statement for monotonically increasing as well as decreasing sequences, this yields 8 cases that must be checked.
\end{remark}

\section{Ehrhart quasi-polynomials for $bD_x$ and $cD_x$}
\label{sec:Ehrhart}

Our first goal in this section is to give an expression for the number of lattice points in the polytopes $P_{nbD_x}$ and $P_{ncD_x}$.

\begin{proposition}
\label{prop:Ehrhart}
Keep the notation and hypotheses of Theorem \ref{c:P4bc-conjs}, and let $\delta\in\{b,c\}$. Then
\[
|P_{n\delta D_x} \cap \mathbb{Z}^2|=c_2(\delta D_x,n)n^2+c_1(\delta D_x,n)n+c_0(\delta D_x,n),
\]
where the $c_i$'s are given as follows.

\begin{enumerate}
\item Set $s = \frac{b}{c}$. For $\delta=b$, we have $c_2(bD_x,n)=2s$, $c_1(bD_x,n)=\frac{1}{2}( 1 + s + \frac{4}{c} )$, and
\begin{align*}
c_0(bD_x,n)
&= 1 - \frac{1}{8s}\left(\left\{ 4sn \right\}^2  - \left\{ 4sn \right\}\right) -\frac{5}{2}\left\{ sn \right\} + \sum_{j=0}^{4 \left\{ \frac{\lfloor 4sn \rfloor}{4} \right\}} \left\{ \frac{3}{4}j \right\} \\
&+ \frac{b-1}{2} \left\{ \frac{4n}{c} \right\} 
- \sum_{j=0}^{b\left\{\frac{\lfloor 4sn \rfloor}{b}\right\}} \left\{ \frac{-p}{b}j \right\}.
\end{align*}

\item For $\delta=c$, we have $c_2(cD_x,n) = \frac{2}{s}$, $c_1(cD_x,n) = \frac{1}{2}( 1 + \frac{1}{s} + \frac{4}{b})$, and
\[
c_0(cD_x,n) = 1 - \left\{ \frac{4n}{b} \right\} - \frac{b-1}{2} \left\{ \frac{4n}{b} \right\} + \sum_{j=0}^{b\{\frac{4n}{b}\}} \left\{ \frac{-p}{b}j \right\}.
\]
\end{enumerate}
\end{proposition}
\begin{proof}
First consider $\delta = b$.
\[
|P_{nbD_x} \cap \mathbb{Z}^2| := \left|\Conv(A,B,C)\right| := \left|\Conv\left((0,0),(-n,0),\left(-3n\frac{b}{c},4n\frac{b}{c}\right)\right)\right|
\]
where $BC$ is given by $y=\frac{b}{p}x+\frac{b}{p}n$ and $AC$ is given by $y=-\frac{4}{3}x$. We will compute $|P_{nbD_x} \cap \mathbb{Z}^2|$ by counting the number of lattice points lying on each line segment $A_jB_j$, where $A_j = (-\frac{3}{4}j,j)$ lies on $AC$ and $B_j = (\frac{p}{b}j - n, j)$ lies on $BC$, for $j = 0,1,\dots,\lfloor 4sn \rfloor$. Here, our approach is similar to that of \cite[Theorem 3.1]{L11}. Denote $M=\lfloor 4sn \rfloor = 4sn - \{ 4sn \}$. Then,
\begin{align*}
|P_{nbD_x} \cap \mathbb{Z}^2| &= \sum_{j=0}^M \left( \left\lfloor -\frac{3}{4}j \right\rfloor - \left\lceil \frac{p}{b}j-n \right\rceil + 1 \right)  \\
&= (n+1)(M+1) + \sum_{j=0}^M \left( -\left\lceil \frac{3}{4}j \right\rceil + \left\lfloor \frac{-p}{b}j \right\rfloor \right) \\
&= (n+1)(M+1) + \sum_{j=0}^M \left( - \left( \frac{3}{4}j + 1 - \left\{ \frac{3}{4}j \right\} \right) +\frac{-p}{b}j  - \left\{ \frac{-p}{b}j \right\} \right) + \left\lfloor \frac{M}{4} \right\rfloor + 1 \\
&= n(M+1) + \frac{M}{4} - \left\{ \frac{M}{4} \right\} + 1 + \sum_{j=0}^M \left( \frac{-p}{b}-\frac{3}{4} \right) j  + \sum_{j=0}^M \left\{ \frac{3}{4}j \right\} - \sum_{j=0}^M \left\{ \frac{-p}{b}j \right\}. \\
\end{align*}
Rewrite the sums involving fractional parts as sums of a linear term in $n$ and a $c-$periodic term in $n$:
\begin{align*}
\sum_{j=0}^M \left\{ \frac{3}{4}j \right\}
&= \left\lfloor \frac{M}{4} \right\rfloor \sum_{j=0}^3 \left\{ \frac{3}{4}j \right\} + \sum_{j=0}^{4\left\{ \frac{M}{4}\right\}} \left\{ \frac{3}{4}j \right\} 
= \frac{3}{2}\left(sn - \left\{ sn \right\}\right) + \sum_{j=0}^{4 \left\{ \frac{\lfloor 4sn \rfloor}{4} \right\}} \left\{ \frac{3}{4}j \right\}, \\
\sum_{j=0}^M \left\{ \frac{-p}{b}j \right\}
&= \left\lfloor \frac{M}{b} \right\rfloor \sum_{j=0}^{b-1} \left\{ \frac{-p}{b}j \right\} + \sum_{j=0}^{b\left\{\frac{M}{b}\right\}} \left\{ \frac{-p}{b}j \right\} 
= \frac{b-1}{2}\left(\frac{4n}{c} - \left\{ \frac{4n}{c} \right\}\right) + \sum_{j=0}^{b\left\{\frac{\lfloor 4sn \rfloor}{b}\right\}} \left\{ \frac{-p}{b}j \right\},
\end{align*}
where we have used the identity
\[
\sum_{j=0}^{b-1} \left\{ \frac{-p}{b}j \right\} = \frac{b-1}{2}
\]
given that $b$ and $p$ are coprime. Moreover, by a direct computation,
\[
\sum_{j=0}^M \left( \frac{-p}{b}-\frac{3}{4} \right) j =
-\binom{M+1}{2}\frac{c}{4b} 
= -2sn^2 + \left(\left\{ 4sn \right\} - \frac{1}{2}\right)n - \frac{1}{8s}\left(\left\{ 4sn \right\}^2  - \left\{ 4sn \right\}\right). \\
\]
Thus, we can write $|P_{nbD_x} \cap \mathbb{Z}^2| = c_2(bD_x,n)n^2 + c_1(bD_x,n)n+c_0(bD_x,n)$, where
\begin{align*}
c_2(bD_x,n) &= 
2s, \\
c_1(bD_x,n) &= 
\frac{1}{2}\left( 1 + s + \frac{4}{c} \right), \\
c_0(bD_x,n) 
&= 1 - \frac{1}{8s}\left(\left\{ 4sn \right\}^2  - \left\{ 4sn \right\}\right) -\frac{5}{2}\left\{ sn \right\} + \sum_{j=0}^{4 \left\{ \frac{\lfloor 4sn \rfloor}{4} \right\}} \left\{ \frac{3}{4}j \right\} + \frac{b-1}{2} \left\{ \frac{4n}{c} \right\} - \sum_{j=0}^{b\left\{\frac{\lfloor 4sn \rfloor}{b}\right\}} \left\{ \frac{-p}{b}j \right\}.
\end{align*}

Likewise, we may find the Ehrhart quasi-polynomial for $|P_{ncD_x} \cap \mathbb{Z}^2|$. To simplify our calculations, we may consider $naD_z = 4nD_z \sim ncD_x$. By linear equivalence, $|P_{naD_z} \cap \mathbb{Z}^2| = |P_{ncD_x} \cap \mathbb{Z}^2|$. The polytope of $naD_z$ is given by 
\[
P_{naD_z} = \Conv(A',B',C') := \Conv\left((0,0),\left(-\frac{4p}{b}n,-4n\right), (3n,-4n)\right),
\]
with $A'C'$ contained in the line $y=-\frac{4}{3}x$ and $A'B'$ contained in the line $y=\frac{b}{p}x$. 
Similarly as before:
\begin{align*}
|P_{naD_z} \cap \mathbb{Z}^2| 
&= \sum_{j=0}^{4n} \left( \left\lfloor \frac{3}{4}j \right\rfloor - \left\lceil \frac{-p}{b}j \right\rceil + 1 \right) \\
&= \sum_{j=0}^{4n} \left( \frac{3}{4}j - \left\{ \frac{3}{4}j \right\} \right) - \sum_{j=0}^{4n} \left( \frac{-p}{b}j + 1 - \left\{ \frac{-p}{b}j \right\} \right) + \left\lfloor \frac{4n}{b} \right\rfloor + 1 + 4n + 1 \\
&= \frac{c}{4b}\binom{4n+1}{2} - \frac{3}{2}n + \frac{2(b-1)}{b}n - \frac{b-1}{2} \left\{ \frac{4n}{b} \right\} + \sum_{j=0}^{b\{\frac{4n}{b}\}} \left\{ \frac{-p}{b}j \right\} + \frac{4n}{b} + 1 - \left\{ \frac{4n}{b} \right\},
\end{align*}
using expressions that we obtained previously. Thus, we can write $|P_{ncD_x} \cap \mathbb{Z}^2| = |P_{naD_z} \cap \mathbb{Z}^2| = c_2(cD_x,n)n^2 + c_1(cD_x,n)n + c_0(cD_x,n)$, where
\begin{align*}
c_2(cD_x,n) &= \frac{2}{s}, \\
c_1(cD_x,n) &= \frac{1}{2}\left( 1 + \frac{1}{s} + \frac{4}{b} \right) n, \\
c_0(cD_x,n) &= 1 - \left\{ \frac{4n}{b} \right\} - \frac{b-1}{2} \left\{ \frac{4n}{b} \right\} + \sum_{j=0}^{b\{\frac{4n}{b}\}} \left\{ \frac{-p}{b}j \right\}.\qedhere
\end{align*}
\end{proof}

Our next goal is to give upper bounds on the constant terms of the Ehrhart quasi-polynomials of $|P_{n\delta D_x} \cap \mathbb{Z}^2|$, $\delta = b, c$. In Proposition \ref{prop:Ehrhart}, notice that the expressions of the last two terms of $c_0(bD_x,n)$ and $c_0(cD_x,n)$ are of the same form, which we will analyze in depth in Section \ref{sec:last2terms}. In the following, we give a uniform upper bound on $c_0(bD_x,n)$ minus its last two terms.
\begin{lemma}
\label{l:3rd-n-4th-term}
In the expression of $c_0(bD_x,n)$, we have
\[
-\frac{5}{2}\left\{ sn \right\} + \sum_{j=0}^{4 \left\{ \frac{\lfloor 4sn \rfloor}{4} \right\}} \left\{ \frac{3}{4}j \right\}\leq\frac{1}{8}
\]
for all $n \geq 0$, where $s = \frac{b}{c}$. Furthermore:
\begin{enumerate}
\item The above expression is positive if and only if $\frac{1}{4}<\{sn\}<\frac{3}{10}$.
\item The above expression is greater than $\frac{-1}{32s}$ only if $\{ sn \} < \frac{1}{2} + \frac{1}{80s}$.
\end{enumerate}
\end{lemma}
\begin{proof}
Let $bn=mc+r$ with $0\leq r\leq c-1$, 
so that $\{sn\}=\frac{r}{c}$. Let $\ell = 0,1,2,3$ be the integer such that $\frac{\ell c}{4}\leq r<\frac{(\ell+1)c}{4}$. Then $\lfloor 4sn\rfloor=4m+\ell$ and so $4 \left\{ \frac{\lfloor 4sn \rfloor}{4} \right\}=\ell$.

Now, we will bound the given expression from the above for each $\ell = 0, 1, 2, 3$. If $\ell=0$, then the given expression in the lemma is $-\frac{5r}{2c}\leq0$. If $\ell=1$, we have $-\frac{5r}{2c}+\frac{3}{4}\leq-\frac{5}{2}\cdot\frac{1}{4}+\frac{3}{4}=\frac{1}{8}$. If $\ell=2$, we have $-\frac{5r}{2c}+\frac{3}{4}+\frac{2}{4}\leq-\frac{5}{2}\cdot\frac{1}{2}+\frac{5}{4}=0$. Lastly, if $\ell=3$, we have $-\frac{5r}{2c}+\frac{3}{4}+\frac{2}{4}+\frac{1}{4}\leq-\frac{5}{2}\cdot\frac{3}{4}+\frac{3}{2}=-\frac{3}{8}$.

For the final statements of the lemma, we see the expression is non-positive if $\ell\neq1$, and so we must have $\frac{1}{4}<\{sn\}$. When $\ell=1$, we computed the expression is equal to $-\frac{5r}{2c}+\frac{3}{4}$, which is positive if and only if $\{sn\}=\frac{r}{c}<\frac{3}{10}$. 

Similarly, the expression could be greater than $\frac{-1}{32s}$ in cases $\ell = 0, 1, 2$. (Note that $s<1$ because $b<c$.)  Working case by case with the expressions obtained, we obtain that $\{ sn \} = \frac{r}{c} < \frac{1}{2} + \frac{1}{80s}$ in order for the expression to be greater than $\frac{-1}{32s}$.
\end{proof}

Note that $-\frac{1}{8s}(\{ 4sn \}^2-\{ 4sn \})\leq\frac{1}{32s}$ since the function $x-x^2$ is maximized at $x=\frac{1}{2}$. Combining this observation with Lemma \ref{l:3rd-n-4th-term}, we obtain the following corollary.

\begin{corollary}
\label{l:1st-through-4th-terms}
We have
\[
c_0(bD_x,n)\leq \frac{9}{8}+\frac{1}{32s}+\frac{b-1}{2} \left\{ \frac{4n}{c} \right\} - \sum_{j=0}^{b\left\{\frac{\lfloor 4sn \rfloor}{b}\right\}} \left\{ \frac{-p}{b}j \right\}.
\]
Moreover, if  $\{ sn \} \geq \frac{1}{2} + \frac{1}{80s}$, we may improve the above bound as follows:
\[
c_0(bD_x,n)\leq 1+\frac{b-1}{2} \left\{ \frac{4n}{c} \right\} - \sum_{j=0}^{b\left\{\frac{\lfloor 4sn \rfloor}{b}\right\}} \left\{ \frac{-p}{b}j \right\}.
\]
\end{corollary}

\section{Bounding $c_0(bD_x,n)$ and $c_0(cD_x,n)$}
\label{sec:last2terms}

In this section, we prove the key results needed to bound $c_0(bD_x,n)$ and $c_0(cD_x,n)$. This amounts to obtaining bounds for the expression $\frac{b-1}{2} \left\{ \frac{4n}{c} \right\} - \sum_{j=0}^r \{ \frac{-p}{b}j \}$, where $r=b\{\frac{\lfloor 4sn \rfloor}{b}\}$. We begin by recording the following lemma.

\begin{lemma}
\label{l:reduce-to-bounding-rhs}
Let $n,b,c\in\ZZ^+$ with $4<b<c$ and $\gcd(4,b,c)=1$. Let $p<0$ be an integer satisfying $4p+3b=c$ and $s=\frac{b}{c}$. If $r=b\left\{\frac{\lfloor 4sn \rfloor}{b}\right\}$, then
\begin{equation*}
\frac{b-1}{2} \left\{ \frac{4n}{c} \right\} - \sum_{j=0}^r\left\{ \frac{-p}{b}j \right\} \leq\frac{(b-1)(r+1)}{2b} - \sum_{j=0}^{r} \left\{ \frac{-p}{b}j \right\}.
\end{equation*}
\end{lemma}
\begin{proof}
Since $r$ is the reminder when $b$ is divided into $\lfloor 4sn\rfloor$, we have $4sn = \lfloor 4sn\rfloor + \{4sn\}=b\lfloor\frac{\lfloor 4sn \rfloor}{b}\rfloor + r+\{4sn\}$. So, 
\[
\left\{ \frac{4n}{c} \right\}=\left\{\frac{4sn}{b}\right\}=\left\{\left\lfloor\frac{\lfloor 4sn \rfloor}{b}\right\rfloor + \frac{r}{b}+\frac{\{4sn\}}{b}\right\}=\frac{r+\{4sn\}}{b},
\]
where the last equality uses $0\leq\frac{r}{b}\leq\frac{b-1}{b}$ and $0\leq\frac{\{4sn\}}{b}<\frac{1}{b}$.
\end{proof}

By the above lemma, it suffices to bound the expression on the righthand side. In \S\ref{subsec:algorithm-for-bounds}, we give a general algorithm to obtain bounds on expressions of the form $\frac{(\beta-1)(u+1)}{2\beta} - \sum_{j=0}^{u} \{ \frac{\alpha}{\beta}j\}$ when $\alpha$ and $\beta$ satisfy particular Diophantine equations, see Corollary \ref{cor:last-two-terms-bd-iter}.

\subsection{An algorithm to bound expressions of the form $\frac{(\beta-1)(u+1)}{2\beta} - \sum_{j=0}^{u} \{ \frac{\alpha}{\beta}j\}$}
\label{subsec:algorithm-for-bounds}

Our goal in this subsection is to prove

\begin{proposition}
\label{prop:specific-cor-last-two-terms-bd-iter}
Suppose that $\alpha_0 > \alpha_1$, $\beta_0 > \beta_1$, and
\[
\alpha_1\beta_0 - \beta_1\alpha_0 = \sigma = \pm 1
\]
where $\alpha_i, \beta_i \in \mathbb{Z}^+$. Let $u_0 = \beta_1t_1 + u_1$, where $u_0,u_1,t_1\in\ZZ^{\geq 0}$ and $0 \leq u_i < \beta_i$ for all $i$. Then 
\[
\frac{(u_0+1)(\beta_0-1)}{2\beta_0} - \sum_{j=0}^{u_0}\left\{ \frac{\alpha_0}{\beta_0}j \right\}= \frac{(u_1+1)(\beta_1-1)}{2\beta_1} - \sum_{j=0}^{u_1} \left\{ \frac{\alpha_1 j}{\beta_1} \right\} +\epsilon(\sigma,t_1,u_1,\beta_0,\beta_1),
\]
where
\[
\epsilon(\sigma,t,u,\beta',\beta)=\frac{(u+1)(\sigma u+\beta'-\beta)}{2\beta'\beta}+\frac{\sigma t(\beta (t-\sigma) + 2u+1-\beta')}{2\beta'}+\Delta(\sigma,t,u,\beta',\beta)
\]
and $\Delta(\sigma,t,u,\beta',\beta)=1$ if $\beta'-\beta\leq \beta t+u$ and $\sigma = -1$, and $\Delta(\sigma,u,\beta',\beta)=0$ otherwise.
\end{proposition}

When applied iteratively, we arrive at the following algorithm.

\begin{corollary}
\label{cor:last-two-terms-bd-iter}
Suppose we have sequences of positive integers $\alpha_0 > \alpha_1 > \dots > \alpha_N$ and $\beta_0 > \beta_1 > \dots > \beta_N$ such that for all $i$, 
\[
\alpha_i\beta_{i-1} - \beta_i\alpha_{i-1} = \sigma_i = \pm 1.
\]
Let $u_0,\dots,u_N$ and $t_1,\dots,t_N$ be non-negative integers satisfying
\[
u_{i-1} = \beta_it_i + u_i
\]
and $0 \leq u_i < \beta_i$ for all $i$. Then
\[
\frac{(u_0+1)(\beta_0-1)}{2\beta_0} - \sum_{j=0}^{u_0}\left\{ \frac{\alpha_0}{\beta_0}j \right\}= \frac{(u_N+1)(\beta_N-1)}{2\beta_N} - \sum_{j=0}^{u_N} \left\{ \frac{\alpha_N j}{\beta_N} \right\} + \sum_{i=1}^N \epsilon(\sigma_i,t_i,u_i,\beta_{i-1},\beta_i)
\]
where $\epsilon$ is as in Proposition \ref{prop:specific-cor-last-two-terms-bd-iter}.
\end{corollary}

We begin with the following preliminary lemmas.

\begin{lemma}
\label{l:sum-of-floors-ceilings-rel-prime}
Let $\alpha,\beta\in\ZZ^+$ be relatively prime. Then
\[
\sum_{k=0}^{\beta-1}\left\lfloor\frac{\alpha}{\beta}k\right\rfloor=\frac{1}{2}(\alpha-1)(\beta-1)\quad\quad\textrm{and}\quad\quad
\sum_{k=0}^{\beta-1}\left\lceil\frac{\alpha}{\beta}k\right\rceil=\frac{1}{2}(\alpha+1)(\beta-1).
\]
\end{lemma}
\begin{proof}
Notice that $\sum_{k=0}^{\beta}\lfloor\frac{\alpha}{\beta}k\rfloor+\beta+1$ is the number of lattice points in the triangle with vertices $0$, $(\beta,0)$, and $(\beta,\alpha)$. So, by Pick's Theorem, 
\[
\sum_{k=0}^{\beta}\left\lfloor\frac{\alpha}{\beta}k\right\rfloor+\beta+1=\frac{1}{2}(\alpha\beta+\alpha+\beta+1)+1.
\]
Since $\sum_{k=0}^{\beta}\lfloor\frac{\alpha}{\beta}k\rfloor=\sum_{k=0}^{\beta-1}\lfloor\frac{\alpha}{\beta}k\rfloor + \alpha$, the first result follows. The second result follows from the first and the fact that $\sum_{k=0}^{\beta-1}\lceil\frac{\alpha}{\beta}k\rceil=(\beta-1)+\sum_{k=0}^{\beta-1}\lfloor\frac{\alpha}{\beta}k\rfloor$.
\end{proof}

\begin{lemma}
\label{l:frac-parts-loewr-sum-new}
Suppose that $\alpha_0 > \alpha_1$, $\beta_0 > \beta_1$, and
\[
\alpha_1\beta_0 - \beta_1\alpha_0 = \sigma = \pm 1
\]
where $\alpha_i, \beta_i \in \mathbb{Z}^+$. Then
\begin{enumerate}
\item\label{l:frac-parts-loewr-sum-new::frac-parts-easy}
$\{\frac{\alpha_0j}{\beta_0} \} = \{ \frac{\alpha_1 j}{\beta_1} \} -\sigma \frac{j}{\beta_1 \beta_0}$ for all $0\leq j<\beta_1$, and
\item\label{l:frac-parts-loewr-sum-new::partial-sum}
for any integer $u$ satisfying $0\leq u<\beta_1$, 
\[
\sum_{j=0}^{u}\left\{\frac{\alpha_0j}{\beta_0}\right\}=\sum_{j=0}^{u}\left\{\frac{\alpha_1j}{\beta_1}\right\}-\frac{\sigma}{\beta_0\beta_1}{u+1\choose 2}
\]
\item\label{l:frac-parts-loewr-sum-new::full-sum}
\[
\sum_{j=0}^{\beta_1}\left\{\frac{\alpha_0j}{\beta_0}\right\}=\frac{1-\sigma+(\beta_1+\sigma)(\beta_0-\sigma)}{2\beta_0}.
\]
\end{enumerate}
\end{lemma}
\begin{proof}
We begin with the proof of (\ref{l:frac-parts-loewr-sum-new::frac-parts-easy}). The case $j=0$ is clear, so we assume $1\leq j\leq\beta_1-1$. Since $\frac{\alpha_0j}{\beta_0}= \frac{\alpha_1 j}{\beta_1} -\sigma \frac{j}{\beta_1 \beta_0}$, it suffices to show $0\leq \{ \frac{\alpha_1 j}{\beta_1} \} -\sigma \frac{j}{\beta_1 \beta_0}<1$. Since $\alpha_1$ and $\beta_1$ are relatively prime, we see $\frac{1}{\beta_1}\leq\{ \frac{\alpha_1 j}{\beta_1} \}\leq1-\frac{1}{\beta_1}$. The result then follows from the fact that $|\sigma \frac{j}{\beta_1 \beta_0}|<\frac{1}{\beta_0}<\frac{1}{\beta_1}$.

Part (\ref{l:frac-parts-loewr-sum-new::partial-sum}) follows from (\ref{l:frac-parts-loewr-sum-new::frac-parts-easy}) by noting
\[
\sum_{j=0}^{u} \left\{\frac{\alpha_0j}{\beta_0}\right\} = \sum_{j=0}^{u} \left(\left\{\frac{\alpha_1 j}{\beta_1}\right\}  - \sigma\frac{j}{\beta_0 \beta_1}\right) = \sum_{j=0}^{u}\left\{\frac{\alpha_1j}{\beta_1}\right\}-\frac{\sigma}{\beta_0\beta_1}{u+1\choose 2}.
\]

To prove (\ref{l:frac-parts-loewr-sum-new::full-sum}), we use $\sum_{j=0}^{\beta_1} \{\frac{\alpha_0j}{\beta_0}\} = \frac{\sigma+1}{2}-\frac{\sigma}{\beta_0}+\sum_{j=0}^{\beta_1-1} \{\frac{\alpha_0j}{\beta_0}\}$ and part (\ref{l:frac-parts-loewr-sum-new::partial-sum}) to see

\begin{align*}
\sum_{j=0}^{\beta_1} \left\{\frac{\alpha_0j}{\beta_0}\right\} &=\frac{\sigma+1}{2}-\frac{\sigma}{\beta_0}+\sum_{j=0}^{\beta_1-1} \left(\frac{\alpha_1 j}{\beta_1} - \left\lfloor\frac{\alpha_1 j}{\beta_1}\right\rfloor\right)-\frac{\sigma}{\beta_0\beta_1}{\beta_1\choose 2}\\
&=\frac{\sigma+1}{2}-\frac{\sigma}{\beta_0}+\frac{\alpha_1}{\beta_1}{\beta_1\choose 2}-\frac{1}{2}(\alpha_1-1)(\beta_1-1)-\sigma\frac{1}{\beta_0 \beta_1}{\beta_1\choose 2},
\end{align*}
where the second equality uses Lemma \ref{l:sum-of-floors-ceilings-rel-prime}. The result follows by algebraic manipulation.
\end{proof}

The following result is the first step in proving Proposition \ref{prop:specific-cor-last-two-terms-bd-iter}.

\begin{corollary}
\label{cor:specific-cor-last-two-terms-bd-iter-k-error}
With hypotheses as in Proposition \ref{prop:specific-cor-last-two-terms-bd-iter}, we have
\[
\frac{(u_1+1)(\beta_0-1)}{2\beta_0} - \sum_{j=0}^{u_1} \left\{ \frac{\alpha_0 j}{\beta_0} \right\}= \frac{(u_1+1)(\beta_1-1)}{2\beta_1} - \sum_{j=0}^{u_1} \left\{ \frac{\alpha_1 j}{\beta_1} \right\} +\epsilon'(\sigma,t_1,u_1,\beta_0,\beta_1),
\]
where
\[
\epsilon'(\sigma,t,u,\beta',\beta)=\frac{(u+1)(\sigma u+\beta'-\beta)}{2\beta'\beta}.
\]
\end{corollary}
\begin{proof}
By Lemma \ref{l:frac-parts-loewr-sum-new} (\ref{l:frac-parts-loewr-sum-new::partial-sum}),
\[
\frac{(u_1+1)(\beta_0-1)}{2\beta_0} - \sum_{j=0}^{u_1} \left\{ \frac{\alpha_0 j}{\beta_0} \right\}=\frac{(u_1+1)(\beta_0-1)}{2\beta_0} - \sum_{j=0}^{u_1} \left\{ \frac{\alpha_1 j}{\beta_1} \right\}+\frac{\sigma}{\beta_0\beta_1}{u_1+1\choose 2}.
\]
Since
\[
\frac{(u_1+1)(\beta_0-1)}{2\beta_0} + \frac{\sigma}{\beta_0\beta_1}{u_1+1\choose 2}=\frac{(u_1+1)(\beta_1-1)}{2\beta_1} +\epsilon'(\sigma,t_1,u_1,\beta_0,\beta_1),
\]
the result follows.
\end{proof}

The next step in proving Proposition \ref{prop:specific-cor-last-two-terms-bd-iter} is provided by

\begin{corollary}
\label{cor:specific-cor-last-two-terms-bd-iter-h-error}
With hypotheses as in Proposition \ref{prop:specific-cor-last-two-terms-bd-iter}, we have
\[
\frac{(u_0+1)(\beta_0-1)}{2\beta_0} - \sum_{j=0}^{u_0}\left\{ \frac{\alpha_0}{\beta_0}j \right\}= \frac{(u_1+1)(\beta_0-1)}{2\beta_0} - \sum_{j=0}^{u_1} \left\{ \frac{\alpha_0 j}{\beta_0} \right\} +\epsilon''(\sigma,t_1,u_1,\beta_0,\beta_1),
\]
where
\[
\epsilon''(\sigma,t,u,\beta',\beta)=\frac{\sigma t(\beta (t-\sigma) + 2u+1-\beta')}{2\beta'}+\Delta(\sigma,t,u,\beta',\beta)
\]
and $\Delta$ is as in Proposition \ref{prop:specific-cor-last-two-terms-bd-iter}.
\end{corollary}
\begin{proof}
We first claim that if $j\in\ZZ^+$ and $1\leq j<\beta_0$, then 
\begin{equation}\label{eqn:subtractbeta1}
\left\{\frac{\alpha_0(j+\beta_1)}{\beta_0}\right\}=\left\{\frac{\alpha_0j}{\beta_0}\right\}-\frac{\sigma}{\beta_0} - \Delta'(\sigma,j)
\end{equation}
where $\Delta'(\sigma,j)=1$ if $j=\beta_0-\beta_1$ and $\sigma=-1$, and $\Delta'(\sigma,j)=0$ otherwise. Since $\alpha_0\beta_1 \equiv -\sigma$ mod $\beta_0$, to prove our claim, it is enough to show the righthand side of (\ref{eqn:subtractbeta1}) lies in the interval $[0,1)$. Note that $\{\frac{\alpha_0j}{\beta_0}\}-\frac{\sigma}{\beta_0}\in[0,1)$ unless either (i) $\sigma=1$ and $\{\frac{\alpha_0j}{\beta_0}\}=0$, or (ii) $\sigma=-1$ and $\{\frac{\alpha_0j}{\beta_0}\}=\frac{\beta_0-1}{\beta_0}$. Case (i) never occurs since $\gcd(\alpha_0,\beta_0)=1$ and $1\leq j<\beta_0$, so $\alpha_0j$ is not divisible by $\beta_0$. Case (ii) occurs exactly when $\alpha_0 j\equiv -1$ mod $\beta_0$, i.e.~when $j=\beta_0-\beta_1$. This establishes our claim.

Recalling that $u_0=\beta_1 t_1+u_1$, we see from equation (\ref{eqn:subtractbeta1}) and Lemma \ref{l:frac-parts-loewr-sum-new} (\ref{l:frac-parts-loewr-sum-new::full-sum}) that
\begin{align*}
\sum_{j=0}^{u_0}\left\{\frac{\alpha_0j}{\beta_0}\right\} &=\sum_{j=\beta_1 t_1+1}^{u_0}\left\{\frac{\alpha_0j}{\beta_0}\right\}+\sum_{j=1}^{\beta_1 t_1}\left\{\frac{\alpha_0j}{\beta_0}\right\}\\
&=\sum_{j=1}^{u_1}\left\{\frac{\alpha_0j}{\beta_0}\right\}-\sigma\frac{u_1t_1}{\beta_0}+t_1\sum_{j=1}^{\beta_1 }\left\{\frac{\alpha_0j}{\beta_0}\right\}-\sigma\frac{\beta_1 }{\beta_0}{t_1\choose2}-\Delta(\sigma,t_1,u_1,\beta_0,\beta_1)\\
&= \sum_{j=0}^{u_1}\left\{\frac{\alpha_0j}{\beta_0}\right\}-\sigma\frac{u_1t_1}{\beta_0}+t_1\frac{1-\sigma+(\beta_1+\sigma)(\beta_0-\sigma)}{2\beta_0}-\sigma\frac{\beta_1 }{\beta_0}{t_1\choose2}-\Delta(\sigma,t_1,u_1,\beta_0,\beta_1)\\
&=\sum_{j=0}^{u_1}\left\{\frac{\alpha_0j}{\beta_0}\right\}+\frac{(u_0+1)(\beta_0-1)}{2\beta_0} - \frac{(u_1+1)(\beta_0-1)}{2\beta_0} - \epsilon''(\sigma,t_1,u_1,\beta_0,\beta_1).\qedhere
\end{align*}
\end{proof}

We now turn to the main result of this subsection.

\begin{proof}[{Proof of Proposition \ref{prop:specific-cor-last-two-terms-bd-iter}}]
Noting that $\epsilon(\sigma,t_1,u_1,\beta_0,\beta_1)=\epsilon'(\sigma,t_1,u_1,\beta_0,\beta_1)+\epsilon''(\sigma,t_1,u_1,\beta_0,\beta_1)$, we see
\begin{align*}
\frac{(u_0+1)(\beta_0-1)}{2\beta_0} - \sum_{j=0}^{u_0}\left\{ \frac{\alpha_0}{\beta_0}j \right\} &= \frac{(u_1+1)(\beta_0-1)}{2\beta_0} - \sum_{j=0}^{u_1} \left\{ \frac{\alpha_0 j}{\beta_0} \right\} +\epsilon''(\sigma,t_1,u_1,\beta_0,\beta_1)\\
&= \frac{(u_1+1)(\beta_1-1)}{2\beta_1} - \sum_{j=0}^{u_1} \left\{ \frac{\alpha_1 j}{\beta_1} \right\} +\epsilon(\sigma,t_1,u_1,\beta_0,\beta_1)
\end{align*}
where the first equality is by Corollary \ref{cor:specific-cor-last-two-terms-bd-iter-h-error} and the second equality is by Corollary \ref{cor:specific-cor-last-two-terms-bd-iter-k-error}.
\end{proof}

We end the subsection by giving a bound on $\epsilon$.

\begin{lemma}
\label{l:bounds-on-epp-and-eppp}
For $t,u,\beta',\beta\in\ZZ$ such that $0\leq u<\beta<\beta'$ and $0\leq t\leq\lfloor\frac{\beta'-1-u}{\beta }\rfloor$, we have
\[
\frac{-(\beta'+\beta-1)^2+4(\beta'-\beta)}{8\beta'\beta}\leq\epsilon(1,t,u,\beta',\beta)\leq\frac{\beta'-1}{2\beta'}
\]
and
\[
\frac{\beta'-\beta}{2\beta'\beta}\leq\epsilon(-1,t,u,\beta',\beta)\leq
\begin{cases}
\frac{1}{8\beta\beta'}(\beta'-\beta+3)(\beta'-\beta-1), &\text{if } u+\beta t<\beta'-\beta\\
1, &\text{otherwise.}
\end{cases}
\]
Furthermore, letting $v=u+\beta t$,
\[
\epsilon(-1,t,u,\beta',\beta)=-\frac{1}{2\beta'\beta}(v+1)(v+\beta-\beta')+\Delta
\]
and 
\[
\epsilon(1,t,u,\beta',\beta)=\frac{1}{2\beta'\beta}(v+1)(v-\beta'-\beta)+\frac{1}{\beta}(u+1).
\]
As functions of $v$, the former (resp.~latter) is increasing (resp.~decreasing) if and only if $v\leq\frac{1}{2}(\sigma\beta+\beta'-1)$, where $u$ is viewed as a constant in the latter.
\end{lemma}
\begin{proof}
Throughout the proof, we treat $\beta$ and $\beta'$ as fixed constants. Letting $v=u+\beta t$, we find
\[
\eta:=2\beta'\beta\epsilon=\sigma v^2+(\sigma(1-\beta')-\beta)v+\beta'(1+\sigma)u+\beta'-\beta+2\beta'\beta\Delta
\]
Then, the expressions $\epsilon(\pm 1,t,u,\beta',\beta)$ are obtained by substituting $\sigma=\pm 1$ into $\eta$. It suffices to bound $\epsilon$ on the larger region $0\leq t\leq\frac{\beta'-1-u}{\beta}$, where our constraints become $0\leq u\leq\beta-1$ and $u\leq v\leq\beta'-1$.

We first consider the case $\sigma=-1$ and bound $\epsilon$ from the above. Recall that $\Delta=0$ if $v<\beta'-\beta$ and $\Delta=1$ otherwise. Then $\eta=-(v+1)(v+\beta-\beta')+2\beta'\beta\Delta$ has a global maximum at $v_{\max}:=\frac{\beta'-\beta+1}{2}$. Since $0\leq v_{\max}\leq\beta'-\beta$, we see that if $v<\beta'-\beta$, then $\eta(u,v)\leq\eta(0,v_{\max})=\frac{1}{4}(\beta'-\beta+3)(\beta'-\beta-1)$. If, on the other hand, $\beta'-\beta\leq v$, then $\eta(u,v)\leq\eta(0,\beta'-\beta)=2\beta'\beta$. The lower bound of $\epsilon$ is then obtained by calculating $\epsilon$ when $v=0,\beta'-1$ and taking the minimum of the two.

We next consider the case $\sigma=1$ . For fixed $u$, the function $\eta(u,v)$ has a global minimum at $v_{\min}:=\frac{\beta+\beta'-1}{2}$. Since $u\leq\beta-1<v_{\min}\leq\beta'-1$ and $|v_{\min}-(\beta'-1)|<|v_{\min}-(\beta-1)|$, we see $\eta(u,v)\leq\eta(u,u)$. As $\eta(u,u)$ is a quadratic in $u$ with global minimum at $\frac{\beta-\beta'-1}{2}<0$, we find $\eta(u,u)\leq\eta(\beta-1,\beta-1)=\beta(\beta'-1)$. This gives the upper bound on $\epsilon$ when $\sigma=1$, and the lower bound is obtained by substituting $v=v_{\min}=\frac{\beta+\beta'-1}{2}$ and $u=0$ into the expression $\epsilon(1,t,u,\beta',\beta)$.

The final statement concerning where $\epsilon(\sigma,t,u,\beta',\beta)$ is increasing is clear from the expression of $\eta$.
\end{proof}

\section{Proof of Theorem \ref{c:P4bc-conjs}}
\label{sec:proof-of-thm}

We prove Theorem \ref{c:P4bc-conjs} using the procedure outlined in Remark \ref{rmk:red-to-bds-seqs}. Throughout this section, we fix the following notation. Let $I=(\frac{\beta^{(1)}}{\alpha^{(1)}},\frac{\beta^{(2)}}{\alpha^{(2)}})$ be one of the four types of intervals listed in Theorem \ref{c:P4bc-conjs} and let $D_0 \sim n_0\delta D_x$ be as listed in  Theorem \ref{c:P4bc-conjs}, with $\delta \in \{b, c\}$ and $\nu_0 = \nu(D_0)$. For example, if $\beta^{(1)}=16k^2$ and $\alpha^{(1)}=8k^2-4k-1$, then $D_0 \sim (2k+1)bD_x$ and $\nu_0=4k$. Throughout this section, for $\delta'\in\{b,c\}$, we let $|n\delta' D_x|:=|P_{n\delta' D_x}\cap\ZZ^2|$. To prove Theorem \ref{c:P4bc-conjs}, we must show
\begin{equation}
\label{eqn:main-ineq-bD1}
|n\delta' D_x|<{\lceil\frac{\delta'}{\delta}\frac{\nu_0}{n_0}n \rceil+1\choose 2}+1
\end{equation}
for each $n$ whenever $\delta' \neq \delta$, and for each $n$ not a multiple of $n_0$ whenever $\delta' = \delta$. Moreover, for each $n$ a multiple of $n_0$ with $\delta'=\delta$, we need to show
\begin{equation}
\label{eqn:main-ineq-bD1-v2}
|n\delta D_x|<{\frac{\nu_0}{n_0}n+2\choose 2}
\end{equation}
Note that inequality (\ref{eqn:main-ineq-bD1-v2}) implies that $\nu(D_0) \leq \nu_0$ for each $D_0$ as listed in Theorem \ref{c:P4bc-conjs}, by Definition \ref{def:nu-and-gamma-exp}. Combining this with the result $\nu(D_0) \geq \nu_0$ established by Proposition \ref{prop:nu-of-winning-poly}, we may prove the claim $\nu(D_0) = \nu_0$ in Theorem \ref{c:P4bc-conjs}.

By Remark \ref{rmk:red-to-bds-seqs}, it suffices to prove (\ref{eqn:main-ineq-bD1}) and (\ref{eqn:main-ineq-bD1-v2}) for weighted projective surfaces satisfying either $\alpha^{(1)}b+\beta^{(1)}p=1$ or $\alpha^{(2)}b+\beta^{(2)}p=-1$ for $\frac{b}{-p}\in I$. It also suffices to consider $b$ (thus $-p$ and $c$) sufficiently large. We begin with the proof of (\ref{eqn:main-ineq-bD1}), the more challenging of the above two equations:

\begin{theorem}
\label{thm:n0-doesnt-divide-n}
Inequality (\ref{eqn:main-ineq-bD1}) holds for each $n$ whenever $\delta'\neq\delta$, and for each $n$ not a multiple of $n_0$ whenever $\delta' = \delta$.
\end{theorem}
\begin{proof}
First, suppose that $\delta'=\delta$. Thus, we need to consider weighted projective surfaces satisfying $\alpha_1 b + \beta_1 p = \pm 1$ with $\frac{b}{-p} \in I'_{k,\mp}$ or $\frac{b}{-p} \in I''_{k,\mp}$ as listed in Theorem \ref{c:P4bc-conjs}. Notice that $\alpha_1$, $\beta_1$ are the corresponding ones listed in Entries 1 and 2 of Table \ref{table:proof-steps}. For the sake of brevity, we prove the result when $\delta'=\delta=b$, $\beta_1:=\beta^{(1)}=16k^2$ and $\alpha_1:=\alpha^{(1)}=8k^2-4k-1$, over the interval $I'_{k,-}=[\frac{\beta_1}{\alpha_1},\frac{2k-1}{k-1}]$ for $k \geq 2$. The other cases are similar and in fact easier.\footnote{The reason the case considered in this proof is the most difficult is because the upper bounds on $\epsilon$ are weakest when $\sigma=-1$; the case considered here corresponds to Entry 1 of Table \ref{table:proof-steps} which has the most number of $\sigma_i=-1$. In addition, among the entries of the table, Entry 1 has the most number of steps.} By Remark \ref{rmk:red-to-bds-seqs}, it suffices to prove the result for $b$ sufficiently large, where
\[
\alpha_1b-\beta_1(-p)=1.
\]
Let $\alpha_0:=-p$, $\beta_0:=b$, and $$|n\delta' D_x|=c_2n^2+c_1n+c_0$$ where the $c_i$ are given as in Section \ref{sec:Ehrhart}.

We begin by giving an upper bound for $c_0$. Letting $r = b\left\{ \frac{\lfloor 4sn \rfloor}{b} \right\}$, we see from Corollary \ref{l:1st-through-4th-terms} and Lemma \ref{l:reduce-to-bounding-rhs} that
\[
c_0\leq\frac{9}{8}+\frac{1}{32s}+\kappa
\]
where $\kappa$ is an upper bound on $\frac{b-1}{2} \{ \frac{4n}{c} \} - \sum_{j=0}^r\{ \frac{-p}{b}j\}$. To obtain such a bound, we apply Corollary \ref{cor:last-two-terms-bd-iter} with $(\alpha_i,\beta_i,\sigma_i)$ given as in Entry 1 of Table \ref{table:proof-steps} below (the other entries listed in the table are used to address the remaining cases whose proof we omit). Note that for each $i\geq1$, we have $\alpha_i\beta_{i-1}-\beta_i\alpha_{i-1}=\sigma_i$. Therefore, if we let $u_0:=r$ and $u_1,\dots,u_5$ and $t_1,\dots,t_5$ be as in Corollary \ref{cor:last-two-terms-bd-iter}, and
\[
\epsilon_i:=\epsilon(\sigma_i,t_i,u_i,\beta_{i-1},\beta_i),
\]
we have
\[
c_0\leq\frac{9}{8}+\frac{1}{32s}+\kappa'+\sum_{i=1}^5\epsilon_i,
\]
where
\[
\kappa':=\frac{u_5+1}{4} - \sum_{j=0}^{u_5} \left\{ \frac{j}{2} \right\}\leq\frac{1}{4}.
\]

\begin{table}[h!]
\centering
\begin{tabular}{ | c | c | c | c | }
\hline
Entry & $\alpha_i$ & $\beta_i$ & $\sigma_i$ \\ \hline
\multirow{5}{*}{$1$} & $\alpha_1 = 8k^2-4k-1$ & $\beta_1 = 16k^2$ & $\sigma_1 = 1$ \\
 & $\alpha_2 = 4k^2-4k+1$ & $\beta_2 = 8k^2-4k+1$ & $\sigma_2 = 1$ \\
 & $\alpha_3 = 4k-3$ & $\beta_3 = 8k-2$ & $\sigma_3 = -1$ \\
 & $\alpha_4 = k-1$ & $\beta_4 = 2k-1$ & $\sigma_4 = -1$ \\ 
  & $\alpha_5 = 1$ & $\beta_5 = 2$ & $\sigma_5 = 1$ \\ \hline
\multirow{5}{*}{$2$} & $\alpha_1 = 8k^2+4k-1$ & $\beta_1 = 4(2k+1)^2$ & $\sigma_1 = 1$ \\
& $\alpha_2 = 4k^2$ & $\beta_2 = 8k^2+4k+1$ & $\sigma_2 = 1$ \\
 & $\alpha_3 = 4k-1$ & $\beta_3 = 8k+2$ & $\sigma_3 = -1$ \\
 & $\alpha_4 = k$ & $\beta_4 = 2k+1$ & $\sigma_4 = 1$ \\ 
 & $\alpha_5 = 1$ & $\beta_5 = 2$ & $\sigma_5 = 1$ \\ \hline
\multirow{3}{*}{$3$} & $\alpha_1 = 2k-1$ & $\beta_1 = 4k$ & $\sigma_1 = 1$  \\
 & $\alpha_2 = k$ & $\beta_2 = 2k+1$ & $\sigma_2 = 1$ \\ 
  & $\alpha_3 = 1$ & $\beta_5 = 2$ & $\sigma_3 = 1$ \\ \hline
\multirow{2}{*}{$4$} & $\alpha_1 = k$ & $\beta_1 = 2k+1$ & $\sigma_1 = 1$ \\ 
 & $\alpha_2 = 1$ & $\beta_2 = 2$ & $\sigma_2 = 1$ \\ \hline
\end{tabular}
\caption{$\alpha_i$ and $\beta_i$ used to bound $\frac{b-1}{2} \{ \frac{4n}{c} \} - \sum_{j=0}^r\{ \frac{-p}{b}j\}$ via Corollary \ref{cor:last-two-terms-bd-iter}, when considering $D \sim nbD_x$. We remark that when considering $D \sim ncD_x$, i.e., bounding $-\frac{b-1}{2} \{ \frac{4n}{b} \} + \sum_{j=0}^r\{ \frac{-p}{b}j\}$ in the constant term of $|P_{ncD_x}\cap\mathbb{Z}^2|$, $\sigma_1 = 1$ needs to be replaced by $\sigma_1=-1$. }
\label{table:proof-steps}
\end{table}

We begin by taking crude upper bounds on the $\epsilon_i$. For small $n$, we will need to replace these crude bounds with more refined ones. By Lemma \ref{l:bounds-on-epp-and-eppp}, we have
\[
\epsilon_i\leq\frac{\beta_{i-1}-1}{2\beta_{i-1}}=:\epsilon_i^+
\]
for $i\in\{1,2,5\}$ and
\[
\epsilon_4\leq1=:\epsilon_4^+.
\]
Furthermore, one checks that for $k\geq11$,
\[
\epsilon_3\leq \frac{1}{8\beta_2\beta_3}(\beta_2-\beta_3+3)(\beta_2-\beta_3-1)=:\epsilon_3^+
\]
as $\epsilon_3^+\geq1$. It is enough to prove Theorem \ref{thm:n0-doesnt-divide-n} for $k\geq11$, leaving the remaining finitely many cases $2\leq k\leq10$ to be checked by hand.


Next, solving for $p$ in terms of $b$, we have $p=\frac{1-\alpha_1b}{\beta_1}$ from which we find $c=4p+3b=\frac{1 + b (1 + 2 k)^2}{4 k^2}$. It follows that $1=\frac{1}{4k^2}\frac{1}{c}+\frac{(1 + 2 k)^2}{4 k^2}s$, and so
\[
\frac{1}{c}=4k^2-(1+2k)^2s.
\]
Note also that from our expression for $c$ in terms of $b$, we have
\[
s=\frac{4 k^2}{(1 + 2 k)^2 + \frac{1}{b}}.
\]
Combining this with our results from Section \ref{sec:Ehrhart}, we see
\[
c_1=\frac{1}{2} + \frac{\beta_1}{\frac{8}{b}+2(3\beta_1 - 4\alpha_1)} + 2(4k^2-(1+2k)^2s).
\]
Recall also that
\[
c_2 = \frac{2\beta_1}{\frac{4}{b}+(3\beta_1-4\alpha_1)}.
\]
We have therefore expressed $c_2$, $c_1$, $s$, $p$, and $c$ all in terms of $k$ and $b$.

Let $n = (2k+1)t + u$ for $t \geq 0$, $1 \leq u \leq 2k$, where $u \neq 0$ since $n_0=2k+1$ does not divide $n$. We can then express
\[
\left\lceil \frac{\nu_0}{n_0}n \right\rceil =\left\lceil \frac{4k}{2k+1}n \right\rceil = 4kt + 2u + \epsilon :=
\begin{cases}
4kt + 2u &\text{if } 1 \leq u \leq k, \\
4kt + 2u - 1 &\text{if } k+1 \leq u \leq 2k.
\end{cases}
\]
We are now ready to show
\[
f := c_2n^2 + c_1n + c_0 - \binom{\left\lceil \frac{\nu_0}{n_0}n \right\rceil+ 1}{2} <1
\]

Replacing $c_0$ by quantity $\frac{9}{8}+\frac{1}{32s}+\frac{1}{4}+\sum_{i=1}^5\epsilon_i^+$, we obtain a larger function $g=\frac{g_1}{g_2}$ where the $g_i$ are polynomials in $t,u,b,k$ and $g_2>0$. One checks that $g_2-g_1$ is decreasing in $t$ and that it is a quadratic in $b$ with positive $b^2$-coefficient for $t>\frac{k}{16}$. Thus, for $t>\frac{k}{16}$ and $b$ sufficiently large, we have shown $f<1$.

We next turn to the case where $t\leq\frac{k}{16}$. Then $n=(2k+1)t+u\leq\frac{k^2}{8}+\frac{33k}{16}$ and so
\[
r\leq 4sn<\frac{\beta_2-\beta_3-1}{2}<\beta_2.
\]
As a result, $r=u_1=u_2$ and $t_1=t_2=0$. We may therefore plug in directly to the definition of $\epsilon_1$ and $\epsilon_2$ to obtain better bounds than $\epsilon_1^+$ and $\epsilon_2^+$; using the final statement of Lemma \ref{l:bounds-on-epp-and-eppp} and the fact that $\frac{\beta_3-\beta_2-1}{2}<0\leq r$, we see
\[
\epsilon_i\leq\frac{1}{2\beta_{i-1}\beta_i}(4sn+1)(4sn+\beta_{i-1}-\beta_i)=:\epsilon_i^{++}
\]
for $i\in\{1,2\}$. Similarly, we find
\[
\epsilon_3\leq-\frac{1}{2\beta_2\beta_3}(4sn+1)(4sn+\beta_3-\beta_2)=:\epsilon_3^{++}.
\]
Using the same argument as in the previous paragraph, replacing the use of $\epsilon_i^+$ with $\epsilon_i^{++}$ for $i\in\{1,2,3\}$, we now find that $g_1-g_2$ is a cubic in $b$ with positive $b^3$-coefficient whenever $n\geq\sqrt{k}$ and $n\neq k+1$.

It therefore remains to handle the cases $n=k+1$ and $n<\sqrt{k}$. We consider $n<\sqrt{k}$ first. Here, $t_1=t_2=t_3=t_4=0$ and $r=u_1=u_2=u_3=u_4$. Furthermore, $r\leq 4sn<\frac{\beta_3-\beta_4-1}{2}$, so we have
\[
\epsilon_4\leq-\frac{1}{2\beta_3\beta_4}(4sn+1)(4sn+\beta_4-\beta_3)=:\epsilon_4^{++}.
\]
Now, since $\beta_5=2$, we know $u_5=0$ or $u_5=1$. Plugging back into the definition of $\kappa'$ and $\epsilon_5$ and using that $\epsilon_5$ is increasing on the range from $r$ to $4sn$, we find
\[
\kappa'+\epsilon_5\leq\epsilon_5^{++}:=
\begin{cases}
\frac{1}{4}+\frac{1}{2\beta_4}((4sn)^2 - (1+\beta_4)(4sn)+\beta_4-2),& r\textrm{\ is even}\\
\frac{1}{2\beta_4}((4sn)^2 - (1+\beta_4)(4sn)+3\beta_4-2),& r\textrm{\ is odd}.
\end{cases}
\]
Treating these cases separately and replacing our use of $c_0$ with $\frac{9}{8}+\frac{1}{32s}+\sum_{i=1}^5\epsilon_i^{++}$ yields $f<1$ for all $n<\sqrt{k}$.

Next, we turn to the case $n=k+1$. Here $r=4k-1$, so $r=u_1=u_2=u_3$, $u_4=u_5=1$, $t_1=t_2=t_3=t_5=0$, and $t_4=2$. Since $\{sn\}=sn-(k-1)\geq\frac{1}{2}+\frac{1}{80s}$, Corollary \ref{l:1st-through-4th-terms} tells us $c_0\leq1+\kappa'+\sum_{i=1}^5\epsilon_i$. Directly using the definition of the $\epsilon_i$ functions, we find $g_1-g_2$ is a quadratic in $b$ with positive $b^2$-coefficient. This concludes our proof for $\delta'=\delta=b$, $\beta_1=\beta^{(1)}=16k^2$ and $\alpha_1=\alpha^{(1)}=8k^2-4k-1$.

Finally, if we suppose $\delta'\neq\delta$ instead, then the weighted projective spaces considered are the ones satisfying $\alpha_1 b+\beta_1 p = \pm 1$ with $\frac{b}{-p}\in I'_{k,\pm}$ or $\frac{b}{-p}\in I''_{k,\pm}$. Notice that $\alpha_1$ and $\beta_1$ are the corresponding ones listed in Entries 3 and 4 of Table \ref{table:proof-steps}, which have considerably fewer steps than the ones for $\delta' = \delta$. The proof is almost exactly the same as the above, with the only difference being the technique used to rewrite $\lceil \frac{\delta'}{\delta}\frac{\nu_0}{n_0}n \rceil$ as a piecewise linear function in $t, u$ such that $n_0t+u = n$, $t \geq 0$, $0 \leq u \leq n_0-1$. We illustrate this with the case $\beta^{(2)}=\beta_1=2k-1, \alpha^{(2)}=\alpha_1=k-1, \sigma_1=-1$ over the interval $I'_{k,-}=[\frac{\beta^{(1)}}{\alpha^{(1)}}, \frac{\beta^{(2)}}{\alpha^{(2)}}]$ for $k \geq 2$, $\alpha^{(1)}$ and $\beta^{(1)}$ as in the previous paragraph. Here, $\delta=b$, $\delta'=c$, and $D_0 \sim (2k+1)bD_x$ with $\nu_0 = 4k$ as before. Notice that
\[
\frac{c}{b} = 4\frac{p}{b}+3 = \frac{4}{\beta^{(2)}}\left(-\alpha^{(2)} - \frac{1}{b}\right) + 3 = \frac{2k+1}{2k-1} - \frac{4}{(2k-1)b},
\]
so that
\[
\left\lceil \frac{c}{b}\frac{4k}{2k+1}n \right\rceil = \left\lceil \frac{4k}{2k-1}n \right\rceil
\]
for all $\frac{16k}{(2k+1)(2k-1)b}n < \frac{1}{2k+1} \iff n < \frac{(2k-1)b}{16k}$. Thus, we may use our previous technique to rewrite the above ceiling function as a polynomial for all $n < \frac{(2k-1)b}{16k}$. Replacing the function $f$ with an upper bound obtained by the same process as before, we may also conclude for $n > \frac{(2k-1)b}{16k}$ by examining the asymptotic behaviour of $f$, similarly as in the previous case. 

We remark that for all other cases where $\delta'\neq\delta$, the ceiling function may be simplified in such a manner for all $n < Cb$ with $C > 0$ a constant.
\end{proof}

To finish the proof of Theorem \ref{c:P4bc-conjs}, we must now handle the case where $n_0$ divides $n$ with $\delta'=\delta$. This is substantially easier than Theorem \ref{thm:n0-doesnt-divide-n}.

\begin{proposition}
\label{prop:n0-does-divide-n}
Inequality (\ref{eqn:main-ineq-bD1-v2}) holds when $n_0$ divides $n$ and $\delta'=\delta$.
\end{proposition}
\begin{proof}
As in the proof of Theorem \ref{thm:n0-doesnt-divide-n}, we handle the case where $\beta_1:=\beta^{(1)}=16k^2$, $\alpha_1:=\alpha^{(1)}=8k^2-4k-1$, $\delta=b$, $n_0=2k+1$, and $\nu_0=4k$. The other cases are similar and easier. By Remark \ref{rmk:red-to-bds-seqs}, it suffices to prove the result for $b$ sufficiently large, where
\[
\alpha_1b-\beta_1(-p)=1.
\]
Let $\alpha_0:=-p$, $\beta_0:=b$, and $(\alpha_i,\beta_i,\sigma_i)$ be as in Table \ref{table:proof-steps}. It is enough to show
\[
f := c_2(n_0t)^2 + c_1n_0t + c_0 - \binom{\nu_0 t + 2}{2} < 0
\]
for all $t \geq 1$. Indeed, if $n=n_0t$ and $|nbD_x| < \binom{\nu_0 n + 2}{2}$, then $\nu(nbD_x) < \nu_0 n + 1$ which implies $\nu(nD_0) \leq \nu_0n$, as required. Replacing $c_0$ by the crude upper bound 
\[
c_0\leq\frac{9}{8}+\frac{1}{32s}+\frac{1}{4}+\sum_{i=1}^5\epsilon_i^+
\]
as in the proof of Theorem \ref{thm:n0-doesnt-divide-n}, we obtain a larger function $g\geq f$. One computes $\frac{\partial g}{\partial t}<0$, so it is enough show $g|_{t=1}<0$. After clearing denominators, one is left with a quadratic in $b$ whose $b^2$-coefficient is negative. Thus, for $b$ sufficiently large, $f\leq g<0$.
\end{proof}

\end{document}


\begin{appendices}
\section{Bounds on $c_0(bD_x,n)$ and $c_0(cD_x,n)$}

For completeness, in this appendix, we list the bounds one uses to carry out the proofs of Theorem \ref{thm:n0-doesnt-divide-n} and Proposition \ref{prop:n0-does-divide-n} for the remaining cases.

\begin{table}[h!]
\centering
\begin{tabular}{ | p{4cm} | c | c | c | c | }
\hline
Phase Transition & Sequence of wps & $\alpha_i$ & $\beta_i$ & $\sigma_i$ \\ \hline
\multirow{8}{*}{$\frac{16k^2}{8k^2-4k-1}, k \geq 2$} & \multirow{4}{*}{$\alpha_1 b - \beta_1 (-p) = 1$} & $\alpha_1 = 8k^2-4k-1$ & $\beta_1 = 16k^2$ & $\sigma_1 = 1$ \\
& & $\alpha_2 = 4k^2-4k+1$ & $\beta_2 = 8k^2-4k+1$ & $\sigma_2 = 1$ \\
& & $\alpha_3 = 4k-3$ & $\beta_3 = 8k-2$ & $\sigma_3 = -1$ \\
& & $\alpha_4 = k-1$ & $\beta_4 = 2k-1$ & $\sigma_4 = -1$ \\ \cline{2-5}
& \multirow{4}{*}{$\alpha_1 b - \beta_1 (-p) = -1$} & $\alpha_1 = 8k^2-4k-1$ & $\beta_1 = 16k^2$ & $\sigma_1 = -1$ \\
& & $\alpha_2 = 4k^2-2$ & $\beta_2 = 8k^2+4k-1$ & $\sigma_2 = -1$  \\
& & $\alpha_3 = 4k-3$ & $\beta_3 = 8k-2$ & $\sigma_3 = -1$ \\
& & $\alpha_4 = k-1$ & $\beta_4 = 2k-1$ & $\sigma_4 = -1$ \\ \hline
\multirow{8}{*}{$\frac{4(2k+1)^2}{8k^2+4k-1}, k \geq 1$} & \multirow{4}{*}{$\alpha_1 b - \beta_1 (-p) = 1$} & $\alpha_1 = 8k^2+4k-1$ & $\beta_1 = 4(2k+1)^2$ & $\sigma_1 = 1$ \\
& & $\alpha_2 = 4k^2$ & $\beta_2 = 8k^2+4k+1$ & $\sigma_2 = 1$ \\
& & $\alpha_3 = 4k-1$ & $\beta_3 = 8k+2$ & $\sigma_3 = -1$ \\
& & $\alpha_4 = k$ & $\beta_4 = 2k+1$ & $\sigma_4 = 1$ \\ \cline{2-5}
& \multirow{4}{*}{$\alpha_1 b - \beta_1 (-p) = -1$} & $\alpha_1 = 8k^2+4k-1$ & $\beta_1 = 4(2k+1)^2$ & $\sigma_1 = -1$ \\
& & $\alpha_2 = 4k^2+4k-1$ & $\beta_2 =8k^2+12k+3$ & $\sigma_2 = -1$ \\
& & $\alpha_3 = 4k-1$ & $\beta_3 = 8k+2$ & $\sigma_3 = -1$ \\
& & $\alpha_4 = k$ & $\beta_4 = 2k+1$ & $\sigma_4 = 1$ \\ \hline
\multirow{4}{*}{$\frac{4k}{2k-1}, k \geq 1$} & \multirow{2}{*}{$\alpha_1 b - \beta_1 (-p) = 1$} & $\alpha_1 = 2k-1$ & $\beta_1 = 4k$ & $\sigma_1 = 1$  \\
& & $\alpha_2 = k$ & $\beta_2 = 2k+1$ & $\sigma_2 = 1$ \\ \cline{2-5}
& \multirow{2}{*}{$\alpha_1 b - \beta_1 (-p) = -1$} & $\alpha_1 = 2k-1$ & $\beta_1 = 4k$ & $\sigma_1 = -1$  \\
& & $\alpha_2 = k$ & $\beta_2 = 2k+1$ & $\sigma_2 = 1$ \\ \hline
\multirow{2}{*}{$\frac{2k+1}{k}, k \geq 1$} & $\alpha_1 b - \beta_1 (-p) = 1$ & $\alpha_1 = k$ & $\beta_1 = 2k+1$ & $\sigma_1 = 1$ \\ \cline{2-5}
& $\alpha_1 b - \beta_1 (-p) = -1$ & $\alpha_1 = k$ & $\beta_1 = 2k+1$ & $\sigma_1 = -1$ \\ \hline
\end{tabular}
\caption{Values of $\alpha_i$ and $\beta_i$ used to simplify $\sum_{j=0}^r\{\frac{-pj}{b}\}$ via Corollary \ref{cor:last-two-terms-bd-iter}.}
\label{table:last-two-terms-bd-iter-steps-old}
\end{table}

\begin{proposition}
\label{prop:const-term-bd-nbD1-A1}
Suppose that $b,p$ coming from $\mathbb{P}(4,b,c)$ with $p < 0$ satisfy
\[
\alpha_1 b - \beta_1(-p) = 1.
\]
Let $r = b\left\{ \frac{\lfloor 4sn \rfloor}{b} \right\} = t_1\beta_1 + u_1 = t_1\beta_1 +t_2\beta_2 +u_2=\dots=\sum_{i=1}^N t_i\beta_i + u_N$, with $\beta_i$'s being respective values from Table \ref{table:last-two-terms-bd-iter-steps} and $u_i \equiv \beta_i, 0 \leq u_i \leq \beta_i - 1$ for each $i$. Then,
\begin{enumerate}
\item For $\alpha_1 = 8k^2+4k-1, \beta_1 = 4(2k+1)^2, k \geq 1$,
\[
c_0(bD_x,n) \leq 
\begin{cases}
-\frac{1}{2\beta_2\beta_3}r^2 + \left( \frac{1}{2\beta_3}-\frac{1}{2\beta_2} \right)r + \frac{15}{8} + \frac{1}{32s} + \frac{k}{2(4k+1)} + \frac{1}{2(2k+1)} + \frac{\beta_2+2}{8\beta_1} &\text{if } r \leq \frac{\beta_2 - \beta_3 - 1}{2}\\
\frac{\beta_2}{8\beta_3} + \frac{25}{8} + \frac{1}{32s} - \frac{1}{2\beta_1} - \frac{1}{2\beta_2} + \frac{k}{2(4k+1)} + \frac{1}{2(2k+1)} &\text{if }r \leq b - 1.
\end{cases}
\]
\item For $\alpha_1 = 2k - 1$, $\beta_1 = 4k, k \geq 1$,
\[
c_0(bD_x,n) \leq 
\begin{cases}
\frac{u_1^2-(2k-1)u_1}{8k} + \frac{15}{8} + \frac{1}{32s} &\text{if }n < \frac{k}{2} \implies r=u_1 = \lfloor 4sn \rfloor < 2k, \\
\leq \frac{15}{8} + \frac{1}{32s} &\text{if }r < b - 1.
\end{cases}
\]
\item For $\alpha_1 = k, \beta_1 = 2k + 1, k \geq 1$,
\[
c_0(bD_x,n) \leq 
\begin{cases}
1 + \frac{3}{2k+1} - \frac{1}{2b} &\text{if }n \leq \frac{k}{2}(1-\frac{1}{40s}) \implies r = u_1 = \lfloor 4sn \rfloor < 2k, \\
\frac{13}{8} + \frac{1}{2(2k+1)} + \frac{1}{32s} - \frac{1}{2b} &\text{if }r \leq b - 1.
\end{cases}
\]
\end{enumerate}
\end{proposition}

\begin{proposition}
\label{prop:const-term-bd-ncD1-A2}
Suppose that $b,p$ coming from $\mathbb{P}(4,b,c)$ with $p < 0$ satisfy
\[
\alpha_1 b - \beta_1(-p) = -1.
\]
Let $r = b\left\{ \frac{\lfloor 4sn \rfloor}{b} \right\} = t_1\beta_1 + u_1 = t_1\beta_1 +t_2\beta_2 +u_2=\dots=\sum_{i=1}^N t_i\beta_i + u_N$, with $\beta_i$'s being respective values from Table \ref{table:last-two-terms-bd-iter-steps} and $u_i \equiv \beta_i, 0 \leq u_i \leq \beta_i - 1$ for each $i$. Then,
\begin{enumerate}
\item For $\alpha_1 = 8k^2-4k-1, \beta_1 = 16k^2, k \geq 2$, 
\[
c_0(cD_x,n) 
\leq \frac{3}{2} - \frac{r}{b} +  \frac{k-1}{2(4k-1)}  + \frac{-u_4^2+(2k-2)u_4}{4(2k-1)} + \frac{\beta_1}{b} + \frac{\beta_2}{\beta_1} + \frac{\beta_3}{\beta_2}.
\]
\item For $\alpha_1 = 8k^2+4k-1, \beta_1 = 4(2k+1)^2, k \geq 1$,
\[
c_0(cD_x,n) \leq  \frac{3}{2} - \frac{r}{b}  + \frac{8k-2}{2(8k+2)} + \frac{-u_4^2+2ku_4}{4(2k+1)} + \frac{\beta_1}{b} + \frac{\beta_2}{\beta_1} + \frac{\beta_3}{\beta_2}.
\]
Moreover, we can refine this upper bound for $r \leq \beta_3 - 1 \implies t_1 = t_2 = t_3 = 0$, $t_4 = 0, 1, 2$:
\[
c_0(cD_x,n) \leq
\begin{cases}
1 + \frac{-u_4^2+(\beta_4-1)u_4}{4\beta_4} + \frac{2k+1}{8\beta_2} \text{ if }t_4 = 0, \\
\frac{3}{2} - \frac{k+1}{8k+2} + \frac{-u_4^2+(\beta_4-1)u_4}{4\beta_4} - \frac{u_4}{\beta_3} + \frac{4k+1}{\beta_2} \text{ if }t_4 = 1, \\
2 - \frac{2k+1}{4k+1} + \frac{-u_4^2+(\beta_4-1)u_4}{4\beta_4} - \frac{2u_4}{\beta_3} + \frac{9k+3}{2\beta_2} \text{ if }t_4 = 2.
\end{cases}
\]
\item For $\alpha_1 = 2k - 1$, $\beta_1 = 4k, k \geq 1$,
\[
c_0(cD_x,n) \leq 
\begin{cases}
1 - \frac{r}{b} &\text{if } r \equiv u_1 \equiv 0 \mod 4k, \\
\frac{7}{4} + \frac{-u_2^2 + 2ku_2}{4(2k+1)} - \frac{1}{4k} + \frac{4k-r}{b}  &\text{if }r < b - 1.
\end{cases}
\]
\item For $\alpha_1 = k, \beta_1 = 2k + 1, k \geq 1$,
\[
c_0(cD_x,n) 
\leq \left( \frac{-1}{4(2k+1)} + \frac{1}{2b(2k+1)} \right)u_1^2 + \left( \frac{1}{4} + \frac{1}{4(2k+1)} + \frac{1}{2b} \right)u_1 + 1 - \frac{r}{b}.
\]
\end{enumerate}
\end{proposition}

\end{appendices}
\end{document}